\documentclass[12pt]{amsart}

\usepackage{verbatim,amsmath,latexsym,amssymb,amsbsy}
\usepackage[leqno]{amsmath}
\usepackage{amsthm}
\usepackage{amscd}
\usepackage{amssymb}
\usepackage{amsfonts}

\newcommand{\m}{\mathfrak m}
\newcommand{\len}{\lambda}
\newcommand{\n}{\mathfrak n}
\newcommand{\J}{\text{Jac}}
\theoremstyle{plain}
\numberwithin{equation}{section}
\newtheorem{theorem}{Theorem}[section]

\newtheorem{discussion}[theorem]{Discussion}
\newtheorem{lemma}[theorem]{Lemma}
\newtheorem{corollary}[theorem]{Corollary}

\newtheorem{definition}[theorem]{Definition}
\newtheorem{remark}[theorem]{Remark}
\newtheorem{example}[theorem]{Example}

\begin{document}

\title{Bounding the first Hilbert coefficient}

\author[Krishna Hanumanthu]{Krishna Hanumanthu}
\address{Department of Mathematics, University of Kansas, Lawrence, KS 66045}
\email{khanuma@math.ku.edu}

\author[Craig Huneke]{Craig Huneke}
\address{Department of Mathematics, University of Kansas, Lawrence, KS 66045}
\email{huneke@math.ku.edu}

\subjclass[2010]{Primary 13A30, 13B22, 13D40, 13H15}
\keywords{}
\thanks{The second author was partially supported by the National Science Foundation, grant
DMS-0756853}

\date{October 28, 2010}
\maketitle

\begin{abstract}
This paper gives new bounds on the first Hilbert coefficient of an ideal of finite
colength in a Cohen-Macaulay local ring. The bound given is quadratic in the multiplicity
of the ideal. We compare our bound to previously known bounds, and give examples to show
that at least in some cases it is sharp. The techniques come largely from work of Elias,
Rossi, Valla, and Vasconcelos.
\end{abstract}

\section{Introduction}

Throughout this paper we study a commutative Noetherian Cohen-Macaulay
local ring $R$ with maximal ideal $\m$ and  infinite residue field $k$. Let
$I$ be an $\m$-primary ideal. We let $d$ be the dimension of $R$, we use $
\mu(\quad)$ to denote the minimal number of generators, and
use $\lambda(\quad)$ to denote the length of an $R$-module.
We can write
$$\lambda(R/I^{n+1}) = e_0(I)\binom{n+d}{d}-e_1(I)\binom{n+d-1}{d-1} + ... +(-1)^de_d(I)$$
where the $e_j(I)$ are integers called the Hilbert coefficients (of $I$),
and the equation is valid for all large values of $n$. There are many studies of 
relationships between the various Hilbert coefficients, and bounds for them. Such bounds are necessary, for example,
to prove finiteness of Hilbert functions of ideals with fixed multiplicity. The work of Srinivas and Trivedi \cite{ST} does
exactly  this.  Northcott \cite{N}
gave one of the first such bounds. He proved that if $R$ and $I$ are as above, then $e_1(I) \geq e_0(I) - \lambda(R/I)$. 
This bound can be improved by adding another term involving the reduction number of $I$.
An ideal $J\subset I$ is a reduction of $I$ if $I^{n+1} = JI^n$ for all large $n$. Let $J$ be a reduction of $I$. The \it reduction
number of $I$ with respect to $J$\rm, denoted $r_J(I)$, is the least integer $r$ such that 	$I^{r+1} = JI^r$. The
least reduction number of $I$ with respect to all reductions is called the \it reduction number of $I$\rm,
denoted $r(I)$. 
Rossi \cite{R} proved if $d\leq 2$, then $e_1(I) \geq e_0(I) - \lambda(R/I) +r(I) -1$.

In this paper we give an upper bound on $e_1(I)$, as opposed to the lower bound given by
Northcott. Several upper bounds exist in the literature. 
In \cite[Prop. 2.10]{RV}, the bound $e_1(I) \leq \binom {e_0(I)-k+1}{2}$ is given for filtrations; in our context, their assumption
is that $I\subset \m^k$.
Elias \cite[Proposition 2.5]{E1} generalized this
bound by proving that  $e_1(I) \le \binom{e_0(I)-k}{2}$ if $I\subset \m^k$ and the integral closure of $I$ is
not the integral closure of $\m^k$. 
Our main result gives another bound in dimension one which can then be extended to higher dimension, and which in general is stronger than
the previously known bounds. As a corollary, our main result can be used to recover, and slightly extend, the result of Elias stated above. 
He also gives other bounds which we compare ours to in Section 3.
Our methods come from those in \cite{RVV}, but with one additional twist.  
We also show our results are sharp by giving a sequence of examples.

\bigskip

\section{Dimension One}

\bigskip
It is well-known that $e_1(I)$ and $e_0(I)$ can be preserved after going modulo general elements of $I$ until one reaches dimension one.
Thus the one-dimensional case is crucial. But one must also possibly preserve other assumptions, and this makes some difficulties in
obtaining optimal results.
Our first theorem handles the case of dimension one.

\begin{theorem}\label{onedim}
Let $(R,\mathfrak{m})$ be a Cohen-Macaulay local ring of dimension one, and let $I \subset R$ be an $\mathfrak{m}-$primary ideal.
Suppose that there exist distinct integrally closed ideals $J_1, ..., J_{k-1}$ such that $\mathfrak{m} \supsetneq J_{k-1} \supseteq J_{k-2} \supseteq .... \supseteq
 J_{1} \supsetneq \bar{I}$, where $\bar{I}$ is the integral closure of $I$. Then 
$e_1(I) \le \binom{e_0(I)-k}{2}$. If $I$ is not integrally closed, $e_1(I) <  \binom{e_0(I)-k}{2}$.
\end{theorem}
\begin{proof}
For simplicity we write $e_j(I) = e_j$.

We have $\lambda(R/\mathfrak{m}) = 1$, $\lambda(\mathfrak{m}/J_{k-1}) \geq 1$,  
$\lambda(J_{i+1}/J_{i}) \geq 1$ for $i=1,..,k-2$, and finally $\lambda(J_{1}/I) \geq 1$.  Thus
$\lambda(R/I) \ge k+1$.

Let $r$ be the reduction number of $I$. 

By the Eakin-Sathaye theorem \cite{ES}, we see that for $n \le r$, 
$I^n$ can not be generated by fewer than $n+1$ elements. So $\lambda(I^n/\mathfrak{m}I^n) \ge n+1$.  Consider
$\lambda(\mathfrak{m}I^n/I^{n+1})$.  We claim this length is at least $k$. To see this, consider the chain of ideals,
$$\mathfrak{m}I^n\supset J_{k-1}I^n\supset J_{k-2}I^n\supset \ldots \supset J_{1}I^n\supset \bar{I}I^n.$$
None of the ideals in this chain are equal; if for example $J_{l+1}I^n = J_{l}I^n$, then the
determinantal trick (see \cite[1.1.8]{SH}) shows that $J_{l+1}$ is integral over $J_{l}$, a contradiction.
Combing the above two inequalities, we obtain that  $\lambda(I^n/I^{n+1}) \geq n+k+1$ for $n \le r$.

We have $e_0n-e_1 = \lambda(R/I^n)$ for $n >>0$. In fact,
this equality holds for $n=r+1$ and  $e_1 = (r+1)e_0-\lambda(R/I^{r+1})$. 
We rewrite this equality in the following way:
$$ e_1 = (e_0 -\lambda(R/I)) + \sum_{n=1}^{r} (e_0 - \lambda(I^n/I^{n+1})).$$
Using the inequalities above, we obtain that
$$e_1\leq e_0 - (k+1) + re_0 - \sum_{n=1}^{r} (n+k+1).$$
We claim that $$e_0 - (k+1) + re_0 - \sum_{n=1}^{r} (n+k+1)\leq \binom{e_0-k}{2},$$ which will
finish the proof.

To simplify this, we set $\alpha = \frac{r}{e_0}$. Note that
$$\sum_{n=1}^{r} (n+k+1) = \frac{(r+k+1)(r+k+2)}{2} - \frac{(k+1)(k+2)}{2}.$$
Now a simple calculation shows that
$$ \binom{e_0-k}{2} - \big{(}e_0 - (k+1) + re_0 - \sum_{n=1}^{r} (n+k+1)\big{)} $$
$$= (e_0(1-\alpha))^2 - e_0(1-\alpha)(2k+3) + (k+1)(k+2).$$
Setting $x = e_0(1-\alpha)$, our claim becomes
$$x^2 - (2k+3)x + (k+1)(k+2) = (x-(k+1))(x-(k+2))\geq 0.$$
This is clear since $x = e_0-r$ is an integer and both terms are therefore both positive or negative.
\end{proof}

We let $K$ be an infinite field for the following examples.
\begin{example}{\rm 
Let $R = K[[x,y]]/(xy^2)$ and let  $I = (x^2,y)$. 
$R$ is a Cohen-Macaulay local ring of dimension 1 and $I$ is $\m-$primary. Moreover, $I$ is an integrally closed ideal and $I \neq \m$. So we can apply Theorem \ref{onedim} with $k=1$. 

We have that $I^n = (x^{2n}, x^{ 2n-2}y, y^n)$. Since $xy^2=0$ in $R$, a product of $n$ elements of $\{x^2, y\}$ is zero if $y$ is taken at least twice and $x^2$ is taken at least once. 

A $K$-basis of $R/I^n$ is
$\{1,x,\ldots ,x^{2n-1},y, \ldots ,y^{n-1},xy,\ldots , x^{2n-3}y\}.$

Hence $\len(R/I^n) = 5n-4$, which implies that $e_0(I) = 5$ and $e_1(I)=4$. 
So the conclusion of Theorem \ref{onedim} holds. Notice also that 
$e_1(I) \nless \binom{e_0(I)-2}{2}$.}
\end{example}

\begin{example}{\rm
$R =  K[[x,y]]/(xy(x-y))$ and $I_k = (x^{k+1},y)$, $k \ge 1$. 

We claim that the reduction number of $I$ is 1. It suffices to prove that $I^2 = zI$ with $z = y+x^{k+1}$.
To prove this we can lift back to $K[[x,y]]$ and prove that $(x^{k+1},y)^2\subset (zy, zx^{k+1}, x^2y-yx^2)$.
In fact it is clear that it suffices to see that $y^2\in (zy, zx^{k+1}, x^2y-yx^2)$. But $y^2(1+x^k) =
zy-x^{k-1}(x^2y-xy^2)$. As $1+x^k$ is a unit, our claim is proved. It is well-known that having reduction number
equal to $1$ implies that $e_0(I) = \len(I/I^2)$. This latter can be easily to be 
$k+3$. Further, $\len(R/I^2) = 2e_0(I)-e_1(I)$. A simple calculation shows that 
$\len(R/I^2) = 2k+4$. So $e_1(I) = 2$. Observe that the ideals $I_k$ are integrally closed for every $k$. This
can be seen, e.g., by going modulo $y$. Therefore we have a chain of distinct integrally closed ideals,
$$\m\supset I_1\supset I_2\supset ...\supset I_{k+1}.$$ 
Applying Theorem~\ref{onedim} yields
$e_1(I) < \binom{e_0(I)-k}{2}$, but $e_1(I) \nless \binom{e_0(I)-k-1}{2}$.
}
\end{example}

In these examples, $e_1(I)$ has the smallest possible value, namely $e_0(I) - \len(R/I)$. 

\begin{example}{\rm 
$R = K[[x,y]]/(x^2y^2)$ and let  $I = (x^2,y)$. 
Then $e_o(I) = e_1(I) = 6$ and $\len(R/I) = 2$. Here we may apply Theorem \ref{onedim} with $k=1$. However, it is true also that $e_1(I) \le \binom{e_0(I)-2}{2}$.
}
\end{example}

Theorem~\ref{onedim} shows that the maximum length of a chain of integrally closed ideals in $R$ which contain a
given $\m$-primary ideal $I$ is important to understand. The next result gives a good lower bound on this number
which depends on the number of maximal ideals in the integral closure of $R$. We apply this in higher dimensions
in Theorem~\ref{higherdim}.

\begin{theorem}\label{reesval} Let $(R,\m,k)$ be a one-dimensional analytically unramified local domain with infinite
residue field $k$. Set $S$ equal to the integral closure of $R$. Then $S$ is a semilocal domain; set $t = \text{dim}_k(S/\text{Jac}(S))$,
where $\text{Jac}(S)$ is the Jacobson radical of $S$. 
Let $I$ be an integrally closed ideal of $R$. Then there exists a chain
of distinct integrally closed ideals, $\m\supset J_{n-1}\supset ...\supset J_0= I$ where $n = \lfloor \frac{\lambda (R/I) -1}{t}\rfloor$.
\end{theorem} 

\begin{proof} Since $R$ is analytically unramified, $S$ is a module-finite extension of $R$ \cite[Theorem 9.2.2]{SH}. Thus
$S$ is a Noetherian semi-local one-dimensional integrally closed domain, and is therefore a PID. Since $k$ is infinite,
we may choose a minimal reduction $x$ of $I$. In this case $xS = IS$ (note that every ideal of $S$ is integrally closed, and
$IS$ is certainly in the integral closure of $xS$, and is thus equal to $xS$). 

Choose $J \ne I$ to be an integrally closed ideal containing $I$ such that $\lambda(J/I)$ is minimal.
To prove the theorem, it suffices to
prove that $\lambda(J/I)\leq t$. To see this, set $J_1 = J$, and define $J_i$ inductively by repeating this step by replacing
$I$ by $J_{i-1}$. Note that $\lambda(J_i/I)\leq it$. Provided 
$J_i \ne \m$ (i.e.,  
$it <  \lambda(R/I)-1$),  we can continue this chain. It follows
that a chain $\m\supset J_{n-1}\supset ...\supset J_0= I$ exists if $n\leq \frac{\lambda (R/I)-1}{t}$. 

Let $y\in R$ be a minimal reduction of $J$. We have then that $yS = JS$. By way of contradiction, suppose that $\lambda(J/I)\geq t+1$, and choose elements $y,z_1,...,z_t\in J$ such that
$I\subset (I,y)\subset (I,y,z_1) =I_1\subset ...\subset (I,y,z_1,...,z_t) = I_t$ is a chain of distinct ideals.
Since all the elements $z_j$ are in the integral closure of $(I,y)$, it follows that $yS = I_1S = ... = I_tS$.
Write $z_j = ys_j$ for some $s_j\in S$, $1\leq j\leq t$. Let $\n_1,...,\n_l$ be the maximal ideals of $S$, so that $\J(S) = \n_1\cap...\cap  \n_l$. 
By assumption, $T = S/(\n_1\cap...\cap  \n_l)\cong k^t$. Let $\alpha_j$ be the image of $s_j$ in $T$. As
$T$ is a $t$-dimensional vector space over $k$, there is a $k$-linear relation $\beta + \gamma_1\alpha_1 + \ldots + \gamma_t\alpha_t = 0$,
where not all of $\beta, \gamma_1,...,\gamma_t$ are zero. 

Choose $r, r_1,...,r_t\in R$ such that $r_i\equiv \gamma_i$ modulo $\m$, and $r\equiv \beta$ modulo $\m$. If any $\gamma_i$ or $\beta$ is
is $0$, we choose the corresponding lift to be $0$ as well.  Note that not all
of $r, r_1,...,r_t$ are in $\m$, so at least one is a unit. By multiplying by $y$,
we obtain that $u = ry+r_1z_1+\ldots + r_tz_t \in (\n_1\cap\cdots\cap \n_l)y\cap R$. Therefore $uS\subsetneq yS$. 

We claim that
$(I,u)S\subsetneq yS$ as well. Clearly $(I,u)S$ is contained in $yS$. To see they are not equal it suffices to prove they are
not equal in some localization $S_{\n_i}$ for some $1\leq i\leq l$. But since $u\in  (\n_1\cap\cdots\cap \n_l)y$, we get equality
in every localization if and only if $IS_{\n_i} = yS_{\n_i}$ for all $i$. This forces $IS = yS$, a contradiction. Hence
$(I,u)S\subsetneq yS$. This shows that the integral closure of $(I,u)$ is strictly inside $J$. 

We also claim that $u\notin I$. This then contradicts the choice of $J$. We may write $u =  ry+r_1z_1+\ldots + r_jz_j$ where $r_j\ne 0$.
By the choice of liftings, it follows that $r_j\notin \m$ and is therefore a unit. If $u\in I$, we obtain that
$z_j\in (I,y,z_1,...,z_{j-1})$, a contradiction.  
\end{proof} 

\begin{definition} {\rm Let $(R,\m,k)$ be a one-dimensional analytically unramified local domain with infinite
residue field $k$. Set $S$ equal to the integral closure of $R$. We define the \it essential rank \rm of $R$ to be
$t = \text{dim}_k(S/\text{Jac}(S))$,
where $\text{Jac}(S)$ is the Jacobson radical of $S$.}
\end{definition}

We think of the essential rank as the possible number of maximal ideals of $S$, even after some finite extension. 

We will now use the chain constructed in the above theorem to bound the first Hilbert coefficient. 

\begin{corollary}\label{dimone}
 Let $(R,\m,k)$ be a one-dimensional analytically unramified Cohen-Macaulay local domain with infinite  
residue field $k$. Let $I$ be an $\m$-primary ideal of $R$ with integral closure $\bar{I}$. Let $t$ be the essential
rank of $R$.
Then
$e_1(I) \le  \binom{e_0(I)-n}{2}$ where $n = \lfloor \frac{\lambda (R/\bar{I}) -1}{t}\rfloor$.
\end{corollary}
\begin{proof}By Theorem \ref{reesval}, we have a chain of distinct integrally closed ideals, $\m\supset J_{n-1}\supset ...\supset J_0= \bar{I}$ 
where $n = \lfloor \frac{\lambda (R/\bar{I}) -1}{t}\rfloor$.
Moreover, $e_0(I) = e_0(\bar{I})$ \cite[Theorem 11.3.1]{SH}. It is then easy to see that $e_1(I) \le e_1(\bar{I})$.
The conclusion now follows from Theorem \ref{onedim}
\end{proof}

\begin{corollary}\label{dimonealgclosed}
 Let $(R,\m,k)$ be a one-dimensional analytically unramified Cohen-Macaulay local domain with algebraically closed 
residue field $k$. Let $I$ be an $\m$-primary ideal of $R$ with integral closure $\bar{I}$. Let $t$ be the
number of distinct maximal ideals of the integral closure of $R$.
Then  
$e_1(I) \le  \binom{e_0(I)-n}{2}$ where $n = \lfloor \frac{\lambda (R/\bar{I}) -1}{t}\rfloor$.
\end{corollary}
\begin{proof} This follows immediately from Corollary~\ref{dimone} by observing that since $k$ is algebraically closed,
the essential rank of $R$ is exactly $t$. 
\end{proof}

\begin{corollary}\label{complete}
 Let $(R,\m,k)$ be a one-dimensional analytically irreducible Cohen-Macaulay local domain with algebraically
closed residue field $k$. Let $I$ be an $\m$-primary ideal of $R$ with integral closure $\bar{I}$. Then  
$e_1(I) \le \binom{e_0(I)-\len(R/\bar{I})+1}{2}$
\end{corollary}
\begin{proof}The completion $\hat{R}$ of $R$ (with respect to $\m$) is a domain and hence the number of minimal prime ideals
of $\hat{R}$ is 1. This shows that the number of maximal ideals of the integral closure of $R$ is 1 \cite[Proposition 4.3.2]{SH}.
The desired inequality follows from Corollary \ref{dimonealgclosed}.
\end{proof}
\vskip.4truein 
\section{Higher Dimension}
\bigskip
In order to extend the above results on bounding the first Hilbert coefficient to higher dimensions, we use 
a construction to reduce to the one-dimensional case. We set up this construction below.

Let $(R,\m,k)$ be an analytically unramified Cohen-Macaulay local domain 
with an infinite
residue field $k$. Let $I$ be an 
$\m$-primary ideal of $R$, with integral closure $\bar{I}$. Let $d$ be 
the dimension of $R$. Choose a minimal reduction $y, x_2,...,x_d$ of $I$. 
This sequence is a regular sequence since $R$ is
Cohen-Macaulay. Set $T = R[\frac{x_2}{y},...,\frac{x_d}{y}]$. By \cite[Corollary 5.5.9]{SH}, 
$T\cong R[T_2,...,T_d]/(yT_2-x_2,...,yT_d-x_d)$. It follows that the extension of $\m$ to $T$ is a height one prime
ideal; set $A = T_{\m T}$. Observe that $A$ is a one-dimensional analytically unramified domain with an infinite residue
field. We let $S$ be the integral closure of $T$. Note that $S$ is also the integral closure
of $R[\frac{I}{y}]$. Set $B = W^{-1}S$, where $W = T\setminus \m T$. Clearly $B$ is the integral closure of $A$.

\begin{lemma}\label{itoh-katz}
Let $(R,\m,k)$ be an analytically unramified Cohen-Macaulay local domain and assume that dimension of $R$ is 
at least 2.  Let $I$ be an integrally 
closed $\m$-primary ideal of $R$ with a minimal reduction $(y,x_2,...,x_d)$. 
Let $T\cong R[T_2,...,T_d]/(yT_2-x_2,...,yT_d-x_d)$ and $A = T_{\m T}$ be as in the above construction. 
Set $J = IA$. Then $J$ is integrally closed, $e_0(I)=e_0(J)$, $e_1(I)=e_1(J)$ and 
$\len(R/I) = \len(A/J)$.
\end{lemma}
\begin{proof} We may assume that $y,x_2,\ldots,x_d$ is a superficial sequence for $I$. 
The fact that $e_0(I)=e_0(J)$ is proved in \cite[Theorem 1.1]{K}. We note that $e_0(I) = e_0((y,x_2,...,x_d))$.  
It also follows from the proof of Theorem 1.1 in \cite{K},
as on page 1022, that the sequence $yT_2-x_2,...,yT_d-x_d\in R[T_2,...,T_d]_{\m R[T_2,...,T_d]}$ is a superficial
sequence of length $d-1$ for $IR[T_2,...,T_d]_{\m R[T_2,...,T_d]}$.  It is well-known that moding out by a superficial 
sequence of length $d-1$ preserves both $e_0$ and $e_1$, for example see \cite[pg 1330]{R}.

To see that $J$ is integrally closed, let $C = R[\frac{x_2}{y}]_{\m[\frac{x_2}{y}]}$.  It follows from 
\cite[Lemma 7]{I} that $\overline{(y,x_2,...,x_d)C} = \overline{(y,x_2,...,x_d)}C$. Since $\overline{(y,x_2,...,x_d)} = I$,
we conclude
$IC = \overline{IC}$.  Thus $IC$ is integrally closed. 
Clearly dim $C = d-1$ and $y, x_3, ..., x_d$ is a system of parameters of $C$. Moreover,
they form a regular sequence.  Hence $C$ is Cohen-Macaulay. It is also clear that $C$ is 
analytically unramified.
If $d-1 \ge 2$, we may apply \cite[Lemma 7]{I} to $C$ and the ideal $(y, x_3,...,x_d)C$ and 
conclude as above that $IR[\frac{x_1}{y},\frac{x_2}{y}]_{\m[\frac{x_1}{y},\frac{x_2}{y}]}$ is integrally
closed. Proceeding this way, we obtain that $J$ is integrally closed.

Let $I' = (y,x_2,...,x_d) \subset R$. $I'T = yT$. 
It remains to prove that $\len(R/I) = \len(A/J)$. However, by \cite[Prop. 8.4.2 (6)]{SH}, it follows that
$\len(R/I) = \len(B/IB)$, where $B = R[T_2,\ldots,T_d]_{\m R[T_2,\ldots,T_d]}$. Since the ideal
$(yT_2-x_2,...,yT_d-x_d)B$ is contained in $IB$, it follows that $\len(B/IB) = \len(A/J)$.
\end{proof}

\begin{definition}{\rm Let 
$(R,\m,k)$ be an analytically unramified Cohen-Macaulay local domain 
with an infinite
residue field $k$.  Let $y,x_1,...,x_d$ be a regular sequence in $R$. 
We define \it the essential rank \rm of $(y,x_2,...,x_d)$ to be the essential rank of the one-dimensional ring $A$ 
constructed above. 
Let $I$ be an $\m$-primary ideal of $R$. Define the \it essential rank \rm 	of $I$ to be the minimum of essential
ranks of minimal reductions of $I$. 
}\end{definition}

\begin{discussion} {\rm It is well-known that the number of maximal ideals in the integral closure of the ring
$A$ in the above definition is exactly the number of Rees valuations of the ideal $I$. Suppose that $\n$
is one of the these maximal ideals in the integral closure $B$ of $A$. Then $B_{\n}$ is a DVR. If for
all such $\n$ the residue field $B/\n$ is purely transcendental over the residue field $k$ of $R$, then
the essential rank of $I$ is \it exactly \rm the number of Rees valuations of $I$. However, this property of
being purely transcendental will
rarely happen. If $I$ is a simple integrally closed ideal in a two-dimensional regular local ring,
then it has exactly one Rees valuation, and the residue field of the valuation is purely transcendental
over the residue field of $R$, provided this latter residue field is algebraically closed. See
\cite[Cor. 4.4]{H} for details. Dale Cutkosky has pointed out to us that at least in dimension
three or higher, or in dimension two for non rational singularities, the residue fields of the Rees valuations are seldom purely transcendental over
the residue field of the base ring. The essential rank is closely related to the degree function introduced
by David Rees \cite{Re}.} 
\end{discussion}

\begin{theorem}\label{higherdim} Let $(R,\m,k)$ be an analytically unramified Cohen-Macaulay local domain with 
infinite  residue field $k$.  Let $I$ be an 
$\m$-primary ideal of $R$, with integral closure $\bar{I}$. Let $t$ denote the essential rank of $I$. Then 
$e_1(I) \le \binom{e_0(I)-n}{2}$ where $n = \lfloor \frac{\lambda (R/\bar{I}) -1}{t}\rfloor$.
\end{theorem} 
\begin{proof}
If $d=$ dim$(R)=1$, then the essential rank of $I$ is the same as the essential rank of $R$. So the result follows from
Corollary \ref{dimone}. Assume that $d \ge 2$. 

Choose a minimal reduction $y,x_2,...,x_d$ of $I$ such that essential rank of $I$ is equal to the essential
rank of $(y,x_2,..,x_d)$. Let $T = R[\frac{x_2}{y},...,\frac{x_d}{y}] \cong  R[T_2,...,T_d]/(yT_2-x_2,...,yT_d-x_d)$
and $A = T_{\m T}$ as in Lemma \ref{itoh-katz}. Let $B$ be the integral closure of $A$. By hypothesis, 
$t = \text{dim}_k(B/\text{Jac}(B))$,
where $\text{Jac}(B)$ is the Jacobson radical of $B$.

By Corollary \ref{dimone} applied to $IA \subset A$, we have $e_1(IA) \le  \binom{e_0(IA)-n}{2}$ 
where $n = \lfloor \frac{\lambda (A/\overline{IA}) -1}{t}\rfloor$. By Lemma \ref{itoh-katz}, 
$e_0(I) = e_0(IA)$ and $e_1(I) = e_1(IA)$. Moreover, $\overline{IA} = \bar{I}A$ and hence
$\len(R/\bar{I}) = \len(A/\overline{IA})$. The conclusion follows. 
\end{proof}

\begin{corollary}\label{boundonessrank}
Let $(R,\m,k)$ be an analytically unramified Cohen-Macaulay local domain with 
infinite residue field $k$. Let $I$ be an 
$\m$-primary ideal of $R$, with integral closure $\bar{I}$.  Let $t$ denote the 
essential rank of $I$. Suppose that 
$e_1(I) > \binom{e_0(I)-k}{2}$ for some integer $k$. Then 
$t \ge \frac{\len(R/\bar{I})-1}{k}$.
\end{corollary}
\begin{proof}
By Theorem \ref{higherdim}, we have 
$\binom{e_0(I)-k}{2} <  \binom{e_0(I)-n}{2}$ with $n = \lfloor \frac{\lambda (R/\bar{I}) -1}{t}\rfloor$. Hence $k >  n$. So $k \ge \frac{\lambda (R/\bar{I}) -1}{t}$.
\end{proof}

\begin{remark}{\rm Corollary~\ref{boundonessrank} can be rephrased to give a lower bound on the essential rank of an
$\m$-primary ideal $I$ in an analytically unramified Cohen-Macaulay local domain $(R,\m)$ with algebraically
closed residue field $k$. Namely, the condition that $e_1(I) \ge \binom{e_0(I)-k}{2}$ in this Corollary gives the smallest such
value of $k$ as a root of a quadratic equation in $k$ with coefficients functions of $e_0(I), e_1(I)$. Solving this
equation gives the following: if $t$ is the essential rank of $I$, then
$$ t\geq \frac {2(\len(R/\bar{I})-1)}{2e_0(I)-1- \sqrt{8e_1(I)+1}}.$$
Note that in general, $e_0(I)$ is larger than $\len(R/\bar{I})$, but we are subtracting $\sqrt{8e_1(I)+1}$ in the denominator.} 
\end{remark}

As mentioned in the introduction, 
Elias \cite{E1} gives several bounds on $e_1(I)$. One of his bounds states that 
that  $e_1(I) \le \binom{e_0(I)-k}{2}$ if $I\subset \m^k$ and the integral closure of $I$ is
not the integral closure of $\m^k$. 
The next corollary recovers this result, and removes the condition that the integral closure of $I$ is not the integral closure of $\m^k$.

\begin{corollary}
Let $(R,\mathfrak{m})$ be a Cohen-Macaulay local ring and let $I \subset R$ be an $\mathfrak{m}-$primary ideal.
Suppose that $I \subset \mathfrak{m}^k$ for some $k \ge 2$. Then
$e_1(I) \le \binom{e_0(I)-k}{2}$.
\end{corollary}
\begin{proof}
We proceed by induction on $d = dim(R)$.

If $d=1$, the result follows from the Theorem~\ref{onedim} provided $\m^k$ is not contained in the
integral closure of $I$ as we have the sequence
$$\mathfrak{m} \supset \overline{{\mathfrak{m}}^2} \supset \ldots \supset \overline{{\mathfrak{m}}^k} \supset \bar{I}$$
in which every term is distinct.

Suppose that $\m^k\subset \bar{I}$.
Then $e_0(I) = ke_0(\m)$. Since $I^n\subset \m^{nk}$, we see that
$$e_0(\m)(nk)-e_1(\m)\leq e_0(I)n-e_1(I)$$ for all large $n$. Cancelling the first terms, it follows that $e_1(I)\leq e_1(\m)$.
If $R$ is regular, then $e_1(\m ) = 0$, and therefore $e_1(I) = 0$. Since $e_0(I) = k$ in this case, we are done by the
convention that $\binom {0}{l} = 1$.

Henceforth we assume that $e_0(\m)> 1$. By \cite[Cor. 3.3]{KM} we know that $e_1(\m)\leq \binom {e_0(\m)}{2}$. Also note that
$\binom{ke_0(\m)-k}{2} = \binom{e_0(I)-k}{2}$. But
$\binom {e_0(\m)}{2}\leq \binom{ke_0(\m)-k}{2}$
is equivalent to the inequality $e_0(\m)\leq k^2e_0(\m)-k^2-k$, which holds since $k\geq 2$ and $e_0(\m)\geq 2$.

Now let $d > 1$.
Let $a \in I$ be a superficial element for $I$. That is, there exists an integer $c \ge 0$
with $(I^{n+1}:a)\cap I^c = I^n$ for all $n \geq c$.
Since depth$(R) = d > 0$ and $I$ is $\m-$primary, $a$ is a nonzero divisor of $R$.

Then $e_0(I) = e_0(I/aR)$ and $e_1(I) = e_1(I/aR)$. Further, dim$(R/aR) = d-1$. Clearly,
$I/aR \subset {(\m/aR)}^k$. We are done by induction.
\end{proof}

\begin{discussion}{\rm In \cite{RV} and \cite{RV2}, the authors give other bounds on $e_1$. Notable among these is the bound  \cite[Theorem 3.2]{RV2} which states that
$$e_1(I)\leq \binom{e_0(I)}{2} - \binom{b}{2}-\lambda(R/I)+1,$$ where $I$ is an $\m$-primary ideal in a Cohen-Macaulay local ring of dimension
$d$, and $b = \mu(I) - d$.  See also \cite{E1} for different versions of this bound. 

It is a little hard to compare our bound with the above bound in general. Just to give one example, consider the ring
$R = k[[t^7,t^8,t^9,t^{10}]]$ and the
ideal $I = (t^9,t^{10},t^{14},t^{15})$. Then $I$ is integrally closed since all the powers of $t$ in $R$ larger than 9 are in $I$.  A few calculations show that $e_0(I) = e_1(I) = 9$, $b = 3$ and $\lambda(R/I) = 3$. So our bound in Theorem \ref{higherdim} gives $e_1(I) \leq 21$, while the above bound is $e_1(I)\leq 31$. The above bound might be better when $I$ has many generators. In the example given, our bound is
good simply because the essential rank is $1$.}
\end{discussion}

We thank the referee for a careful reading of the paper and for
providing new references.

\end{document}